\documentclass[11pt]{amsart}
\usepackage{latexsym,amsfonts,amssymb,amsthm,amsmath,mathrsfs,color,amscd,graphicx,fullpage,parskip,hyperref,bm,bbm,dsfont}

\usepackage{comment}
\includecomment{uncomment}
\usepackage{commath,mathtools,setspace}
\usepackage{paralist}
\usepackage{extarrows}

\excludecomment{comment} 

\newcommand{\bb}[1]{{\mathbb #1}}

\newcommand{\mder}[2]{\frac{\delta #1}{\delta #2}}

\newtheorem{theorem}{Theorem} 

\newtheorem{lemma}[theorem]{Lemma}
\newtheorem{proposition}[theorem]{Proposition}

\newtheorem{example}[theorem]{Example}
\newtheorem{remark}[theorem]{Remark}
\newtheorem{assumption}[theorem]{Assumption}

\numberwithin{equation}{section}
\numberwithin{theorem}{section}

\newcounter{step}
\begin{document}
	
	\title
	{New uniqueness results for a mean field game of controls}
	
	\author{P.~Jameson Graber}
	\thanks{National Science Foundation under NSF Grant DMS-2045027.}
	\address{J. Graber: Baylor University, Department of Mathematics;\\
		Sid Richardson Building\\
		1410 S.~4th Street\\
		Waco, TX 76706\\
		Tel.: +1-254-710- \\
		Fax: +1-254-710-3569 
	}
	\email{Jameson\_Graber@baylor.edu}
	
	\author{Elizabeth Matter}
	\address{E.~Matter: Baylor University, Department of Mathematics;\\
		Sid Richardson Building\\
		1410 S.~4th Street\\
		Waco, TX 76706
	}
	\email{Ellie\_Carr3@baylor.edu}	
	\date{\today}

	\begin{abstract}
        We propose a new approach to proving the uniqueness of solutions to a certain class of mean field games of controls. In this class, the equilibrium is determined by an aggregate quantity $Q(t)$, e.g. the market price or production, which then determines optimal trajectories for agents. Our approach consists in analyzing the relationship between $Q(t)$ and corresponding optimal trajectories to find conditions under which there is at most one equilibrium. We show that our conditions do not match those prescribed by the Lasry-Lions monotonicity condition, nor even displacement monotonicity, but they do apply to economic models that have been proposed in the literature.
	\end{abstract}

	\keywords{mean field games, }

	\maketitle
	
	\tableofcontents
	
	\section{Introduction}

    Mean Field Games (MFG) were introduced in \cite{lasry_mean_2007} to describe Nash Equilibria of a large population of players in differential games. Each individual seeks to optimize their outcome, while playing against the distribution of all the other players. In a mean field game of controls, each player's cost depends on the distribution of both states and controls. One motivating example we consider in this paper is the Cournot mean field game of controls, a model of the production of an exhaustible resource by a continuum of producers. Much recent work has been done with this model, e.g. in \cite{camilli_learning_2024, chan_fracking_2017, chan_bertrand_2014, graber_existence_2018, graber_master_2023, harris_games_2010, jameson_graber_nonlocal_2021, kobeissi_classical_2022, kobeissi_mean_2022}. 

    An important question in mean field games is if there exists a unique Nash Equilibrium. Lasry and Lions introduced the \emph{Lasry-Lions monotonicity condition} as a sufficient condition to show uniqueness of equilibria. This condition has been the most popular criterion for establishing uniqueness in the literature on mean field games. In fact, it has also been shown in situations where the Lasry-Lions condition does not hold, there is non-uniqueness of equilibria for long time horizon \cite{bardi_non-uniqueness_2019}. In such a case, one way to recover uniqueness is by limiting the time horizon. However, in this paper we will focus on showing uniqueness for an arbitrarily long time horizon. For mean field games of controls, Cardaliaguet and Lehalle in \cite{cardaliaguet_mean_2018} proved uniqueness under the Lasry-Lions condition. More recently \emph{displacement monotonicity} has also been used in proving uniqueness (see \cite{gangbo_global_2022, gangbo_mean_2022, meszaros_mean_2024, bensoussan_stochastic_2019, bensoussan_control_2023, ahuja_wellposedness_2016}). In \cite{graber_monotonicity_2023}, the authors introduced two new types of monotonicity, and showed each of the four conditions were in dichotomy with one another. Since none of these conditions can fully describe all classes of mean field games which have a unique equilibrium, it is natural to ask what doors these new monotonicity conditions open that were closed before.

    In this paper we further explore the monotonicity condition $(\Sigma)$ used in \cite{graber_monotonicity_2023}.
     Conceptually, condition $(\Sigma)$ applies to any mean field game in which Nash equilibrium is equivalent to finding a trajectory, which we call here $Q(t)$, that is a fixed point of the map
    \begin{equation*}
    	Q(t) \mapsto \text{distribution of optimal states and controls} \ \mu_Q(t) \mapsto F(\mu_Q(t)).
	\end{equation*}
	(Think, for example, of market price or aggregate production.)
	If the \emph{error map} $E[Q] := Q - F(\mu_Q(\cdot))$ is strictly monotone, then the Nash equilibrium is unique.
     We prove that this monotonicity condition is met under certain assumptions and show that examples under these assumptions do not have to meet the Lasry-Lions or even displacement monotonicity condition. 
     The assumptions are met by a class of examples inspired by the Cournot competition model from \cite{chan_fracking_2017}, in which the Nash equilibrium is given by a market clearing condition.
    
    The rest of this introduction will be spent describing the model we plan to analyze.  In Section \ref{sec: Error Derivative} we will show sufficient conditions under which the error map is monotone. It is convenient to consider separately two cases.
   	In the first case, the Lagrangian does not depend on the position variable, and we find that the game has a special structure that we exploit to prove uniqueness in a surprisingly simple way; the second case is more general. We conclude the paper with the crucial Section \ref{sec: Examples}, which is devoted to examples. We show that mean field games such as the one considered in \cite{chan_fracking_2017} satisfy the $(\Sigma)$ condition but are not Lasry-Lions monotone. This observation leads us to believe that recent results on Cournot competition \cite{jameson_graber_nonlocal_2021, graber_master_2023} could be improved by dispensing with the ``smallness condition'' used to prove uniqueness. Nevertheless, in this paper we deal only with first-order models without boundary conditions or state constraints, thus avoiding the heavy technicalities and focusing solely on the issue of monotonicity.

 \subsection{Description of the Model}
	Our goal is to prove existence and uniqueness of solutions to the following system:
	\begin{subequations} \label{mfg}
		\begin{align}
			\label{hj} -\partial_t u + H(x,D_x u,Q) &= 0, &\text{in}~\bb{R}^d \times (0,T),\\
			\label{fp} \partial_t m - \nabla \cdot \del{D_p H(x,D_x u,Q)} &= 0, &\text{in}~\bb{R}^d \times (0,T),\\
			\label{Q} -\int_{\bb{R}^d} D_p H(x,D_x u(x,t),Q(t))m(x,t)\dif x &= Q(t) &\text{in}~(0,T),\\
			\label{data} m(x,t) = m_0(x), \quad u(x,T) &= g(x), &\text{in}~\bb{R}^d.
		\end{align}
	\end{subequations}
	System \eqref{mfg} is a  type of MFG of controls, inspired by market competition models. From the Lagrangian $L$, we define the Hamiltonian $H$ by 
 \begin{equation}
     H(x,p,Q):=\sup_{v\in\bb{R}^d}\del{p\cdot v-L(x,v,Q)}.
 \end{equation}
	It corresponds to a game in which players seek to minimize
	\begin{equation}
		\int_0^T L\del{x(t),\dot{x}(t),Q(t)}\dif t + g(x(T))
	\end{equation}
	treating the function $Q(t)$ as given.
	Equation \eqref{Q} is the equilibrium condition.
	It says that $Q(t)$ is the average of all optimal velocities played against $Q(t)$ itself (e.g.~the aggregate equilibrium production).
	An alternative way to determine $Q(t)$ is to compute optimal trajectories using the Euler-Lagrange equations:
	\begin{subequations} 
		\label{E-L}
		\begin{align}
		\od{}{t}D_v L\del{x(t),\dot{x}(t),Q(t)} &= D_x L\del{x(t),\dot{x}(t),Q(t)},\\ 
		x(0) &= x_0,\\
		D_v L\del{x(T),\dot{x}(T),Q(T)}  &= -Dg(x(T)).
		\end{align}
	\end{subequations}
	The solution of System \eqref{E-L} will be denoted $x(t;x_0,Q)$. The assumptions in Section \ref{sec: Error Derivative} will ensure that the Euler-Lagrange equations are solvable, in particular the convexity of $L$ with respect to the pair $(x,v)$. We realize such an assumption will limit the generality of the result. Nevertheless, we are able to prove uniqueness in cases that have already been proposed in the literature for which Lasry-Lions and Displacement Monotonicity may not be applicable.

	Note that $x(t;x_0,Q)$ depends on $Q$ is a nonlocal way, i.e.~not on $Q(t)$ but on the entire trajectory $Q:[0,T] \to \bb{R}^d$.
	Equation \eqref{Q} is equivalent to
	\begin{equation} \label{Q xdot}
		Q(t) = -\int_{\bb{R}^d} \dot{x}(t;x_0,Q)\dif m_0(x_0).
	\end{equation}
	Define the error map $E[Q]$ to be
	\begin{equation} \label{error map}
		E[Q](t) := Q(t) + \int_{\bb{R}^d} \dot{x}(t;x_0,Q)\dif m_0(x_0).
	\end{equation}
	This is the map which shows how far $Q$ is from being a fixed point of the right hand side of equation \eqref{Q xdot}. In order to solve System \eqref{mfg}, we first solve \eqref{Q xdot} by arguing that the error map is strictly monotone with respect to $Q$ in $L^2([0,T],\bb{R}^d)$.
	
\section{Monotonicity of the Error Map}\label{sec: Error Derivative}
    \subsection{When the Lagrangian does not depend on $x$}
    We aim to show the monotonicity of $E$ first for the case where there is no $x$ dependence in the Lagrangian. This case has interesting properties and is relevant to the Cournot model where the Lagrangian is $L(v,Q)=vP(-v+\varepsilon Q)$, which we will discuss further in Section \ref{sec: Examples}. We will consider Lagrangians which meet the following assumptions.
    \begin{assumption}\label{Lassum}
        The Lagrangian $L:\bb{R}^d\times\bb{R}^d\to \bb{R}$ is $C^2$ and satisfies the following:
        \begin{itemize}
            \item $L(v,Q)$ is strictly convex with respect to $v$.
            \item $\del{D^2_{vv}L(v,Q)}^{-1}\del{-D^2_{vQ}L(v,Q)} + I>0$ for all $v,Q$. (We say an $n\times n$ matrix $A>0$ in the sense that $x^TAx>0$ for all $x\in\mathbb{R}^n\setminus\cbr{0}$. We denote by $I$ the identity matrix.)
        \end{itemize}
    \end{assumption}
\begin{lemma}\label{lem: xdot}
    If $L$ satisfies \ref{Lassum}, then there exists a unique solution of the Euler-Lagrange equations \eqref{E-L}. Its velocity is given by
    \begin{equation} \label{optimal dotx}
    	\dot{x}(t;x_0,Q)=D_pH(-Dg(x_T(x_0,Q)),Q(t)). 
    \end{equation}
\end{lemma}
\begin{proof}
   If $L$ does not depend on $x(t)$ then \eqref{E-L} becomes
 	\begin{subequations} 
		\label{E-L2}
		\begin{align}
		\od{}{t}D_v L\del{\dot{x}(t),Q(t)} &= 0,\\ 
		x(0) &= x_0,\\
		D_v L\del{\dot{x}(T),Q(T)}  &= -Dg(x(T)).
		\end{align}
	\end{subequations}
 Then $D_v L\del{\dot{x}(t),Q(t)}=-Dg(x(T))$, and because $D^2_{vv}L\del{\dot{x}(t),Q(t)}>0$ 
 \begin{equation}
     \dot{x}(t;x_0,Q)=D_pH(-Dg(x(T)),Q(t)).
 \end{equation}
 Now we show the dependence of $x(T)$ on $x_0$ and $Q$.
\begin{equation}
    x(T)-x(0)=\int_0^T\dot{x}(t;x_0,Q)dt=\int_0^TD_pH(-Dg(x(T)),Q(t))\dif t
\end{equation}
so we define $x_T(x_0,Q)=x(T;x_0,Q)$ implicitly by the following equation
\begin{equation}\label{eq: x_T}
    x_T(x_0,Q)-\int_0^TD_pH(-Dg(x_T(x_0,Q)),Q(t))\dif t=x_0.
\end{equation} 
Note this is solvable as the left side of \eqref{eq: x_T} is a monotone function of $x_T$ by Assumption \ref{Lassum}.
Therefore the optimal velocity played against $Q(t)$ is given by \eqref{optimal dotx}.
\end{proof}

\begin{lemma}\label{lem:Q const}
    Suppose $D^2_{pQ}H(p,Q)+I>0$ for all $p,Q$. Then if $Q(t)$ solves 
    \begin{equation}\label{eq: Q}
        Q(t)=-\int D_pH(-Dg(x_T(x_0,Q)),Q(t))\dif m_0(x_0),
    \end{equation}
    $Q(t)$ is constant.
\end{lemma}
\begin{proof}
    Suppose $Q(t)$ solves equation \eqref{eq: Q}. Since $x_T$ depends on the entire trajectory of $Q$, not $Q(t)$, differentiating with respect to $t$ yields
    \begin{equation}
        \dot{Q}(t)=-\int D^2_{pQ}H(-Dg(x_T(x_0,Q)),Q(t))\dot{Q}(t)\dif m_0(x_0).
    \end{equation}
    This becomes
    \begin{equation}
        \del{I+\int D^2_{pQ}H(-Dg(x_T(x_0,Q)),Q(t))\dif m_0(x_0)}\dot{Q}(t)=0.
    \end{equation}
    So if $D^2_{pQ}H(p,Q)>-I$, then $Q(t)$ must be constant.
\end{proof}
\begin{remark}
    Since $D^2_{pQ}H(p,Q)=\del{D^2_{vv}L(v,Q)}^{-1}\del{-D^2_{vQ}L(v,Q)}>-I$ for $v=D_pH(p,Q)$, the optimal control $Q(t)$ under the assumptions \ref{Lassum} must be constant.
\end{remark}
Now we call $Q(t)=Q$ for constant $Q$. Our next step is to show the monotonicity of the map 
\begin{equation}
   E[Q]:=Q+\int D_pH(-Dg(x_T(x_0,Q)),Q)\dif m_0(x_0).
\end{equation}
 
\begin{lemma}\label{lem: D_Qx_T}
    If $D_{pQ}^2H(p,Q)>-I$ for all $p,Q$, and $L$ satisfies the assumptions of \ref{Lassum}, 
    $$D_Qx_T(x_0,Q)=\del{I+TD^2_{pp}H(-Dg(x_T(x_0,Q)),Q)D^2g(x_T(x_0,Q))}^{-1}\del{TD^2_{pQ}H(-Dg(x_T(x_0,Q)),Q)}.$$
\end{lemma}
\begin{proof}
    By Lemma \ref{lem:Q const}, $Q$ is constant so equation \eqref{eq: x_T} becomes 
    \begin{equation}
        x_T(x_0,Q)-TD_pH(-Dg(x_T(x_0,Q)),Q)=x_0.
    \end{equation}
    Taking the implicit partial derivative with respect to $Q$ yields
    \begin{equation}
        D_Qx_T+TD^2_{pp}H(-Dg(x_T),Q)D^2g(x_T)D_Qx_T-TD^2_{pQ}H(-Dg(x_T),Q)=0.
    \end{equation}
    Solving for $D_Qx_T$ completes the proof.
\end{proof}

\begin{theorem}\label{thm: G mon1}
    Suppose $L$ satisfies the assumptions of \ref{Lassum}, and $D^2g>0$. Then $E$ is monotone in the sense that for all $Q_0,Q_1\in\bb{R}^d$ such that $Q_0\not=Q_1$,
    \begin{equation}
        \langle E[Q_1]-E[Q_0],Q_1-Q_0\rangle>0
    \end{equation}
\end{theorem}
\begin{proof}
    Let $Q_1,Q_0\in \bb{R}^d$. Then 
    \begin{equation}
        \langle E[Q_1]-E[Q_0],Q_1-Q_0\rangle=\langle DE[\tilde{Q}](Q_1-Q_0),Q_1-Q_0\rangle
    \end{equation}
    for some $\tilde{Q}$ between $Q_0$ and $Q_1$. Using Lemma \ref{lem: D_Qx_T},
    \begin{equation}\label{eq: DE}
    \begin{split}
        DE[\tilde{Q}]&=I+\int \del{D_{pQ}^2H(-Dg(x_T),\tilde{Q})-D^2_{pp}H(-Dg(x_T),\tilde Q)D^2g(x_T)D_Qx_T}\dif m_0(x_0) \\
       &=I+\int\del{I+TD^2_{pp}H(-Dg(x_T),\tilde Q)D^2g(x_T)}^{-1}D^2_{pQ}H(-Dg(x_T),\tilde Q)\dif m_0(x_0)
    \end{split}
    \end{equation}
    Now since $D^2_{pQ}H(p,Q)>-I$ and $D^2g(x_T)>0$, 
    \begin{equation}
        \langle E[Q_1]-E[Q_0],Q_1-Q_0\rangle>0
    \end{equation}
    for all $Q_1,Q_0\in \bb{R}^d$ such that $Q_1\not=Q_0$.
\end{proof}

\begin{theorem}
    Suppose $L$ satisfies the assumptions of \ref{Lassum}, $D^2_{vv}L(v,Q)\leq MI$ for some $M<\infty$, and $D^2g\geq cI>0$. Then $E$ is monotone in the sense that for all $Q_1,Q_0\in\mathbb R^d$
\begin{equation}\label{eq: mon}
    \langle E[Q_1]-E[Q_0],Q_1-Q_0\rangle>\del{1-\delta}\enVert{Q_1-Q_0}^2,
\end{equation} 
where $\delta<1$.
\end{theorem}

\begin{proof}
    We make use of the proof of Theorem \ref{thm: G mon1} through equation \eqref{eq: DE}.
    
    Now since $D^2_{vv}L(v,Q)\leq MI$, $D^2_{pp}H(p,Q)\geq \frac{1}{M}I>0$. We also have that $D^2_{pQ}H(p,Q)>-I$ and $D^2g(x_T)\geq cI>0$, so
    \begin{equation}
        \langle E[Q_1]-E[Q_0],Q_1-Q_0\rangle>\del{1-\frac{1}{1+T\frac{c}{M}}}\enVert{Q_1-Q_0}^2.
    \end{equation}
    Thus \eqref{eq: mon} holds with $\displaystyle\delta=\frac{1}{1+T\frac{c}{M}}$.    
\end{proof}

\subsection{When the Lagrangian depends on $x$}
Now we treat the case where the Lagrangian does depend on $x(t)$. As before our goal is to compute the derivative of $E$. However, because of the $x$-dependence, we cannot explicitly solve for the optimal velocity of the Euler-Lagrange equations. Instead we use the following approach:

Suppose that $y(t),\dot{y}(t):[0,T] \to\mathbb{R}^d$ satisfy the following equations:
\begin{equation}
    y(t)=\int_0^T\mder{x}{Q}\del{t,x_0,Q}(s)\tilde Q(s)\dif s
\end{equation}
and
\begin{equation}
    \dot{y}(t)=\int_0^T\mder{\dot{x}}{Q}\del{t,x_0,Q}(s)\tilde Q(s)\dif s.
\end{equation}

Then letting $X(t)=(x(t),\dot{x}(t),Q(t))$ and taking the derivative with respect to $Q$ of \eqref{E-L}, we have the system:
	\begin{equation} \label{eq:E-L3}
        \begin{split}
        \od{}{t}\sbr{D_{vx}^2 L\del{X(t)}y(t)+D_{vv}^2L(X(t))\dot{y}(t)+D_{vQ}^2L(X(t))\tilde Q(t) }&= D_{xx}^2 L\del{X(t)}y(t)+D_{xv}^2L(X(t))\dot{y}(t)\\
        &+D_{xQ}^2L(X(t))\tilde Q(t),\\ 
		y(0) &= 0,\\
		D_{vx}^2 L\del{X(T)}y(T)+D_{vv}^2L(X(T))\dot{y}(T)+D_{vQ}^2L(X(T))\tilde Q(T)  &= -D^2g(x(T))y(T). 
        \end{split}
	\end{equation}
 Now we seek to show the existence and uniqueness of $x(t)$ which solves \eqref{E-L} then use estimates on $y(t)$ from \eqref{eq:E-L3} to tell us about the derivative of the error map. Before we proceed we make the following assumptions on $L$:
 \begin{assumption}\label{Lassum2}
   The Lagrangian $L:\bb{R}^d\times\bb{R}^d\times\bb{R}^d\to\bb{R}$ is $C^2$ and satisfies: 
   \begin{itemize}
       \item $D_xL(x,v,Q)\not=0$
       \item $D^2_{(x,v)}L\geq c>0$.
       \item $c^2>2M^2$, where $M=\max\cbr{\enVert{D_{vQ}^2L(X(t))}_{L_\infty},\enVert{D_{xQ}^2L(X(t))}_{L_\infty}}$.
   \end{itemize}
 \end{assumption}
 \begin{remark}
     Assumption \ref{Lassum2} means that the Lagrangian's dependence on $Q$ does not overpower its convexity with respect to $(x,v)$. This does not necessarily imply small dependence on $Q$, since in some applications $c$ need not be small.
 \end{remark}
 \begin{lemma}\label{thm: well posed}
    If $L$ satisfies Assumption \ref{Lassum2} and $D^2g(x)\geq 0$, there exists a unique $x(t)$ which solves \eqref{E-L}.
    Likewise, given $x(t)$, there exists a unique solution $y(t)$ which solves \eqref{eq:E-L3}.
\end{lemma}
\begin{proof}
 In order to prove this theorem, we make use of results from Chapter 6 of Cannarsa and Sinestrari's book \cite{cannarsa_semiconcave_2004}. The strict convexity of $L$ and $D^2g(x)\geq 0$ imply the criteria that ensure there exists a solution satisfying the Euler-Lagrange Equations. Additionally the strict convexity of $L$ means that such a solution is unique.
\end{proof}
\begin{proposition}
  Suppose $L$ satisfies Assumption \ref{Lassum2} and $D^2g(x)\geq 0$. If $y(t)$ solves the system \eqref{eq:E-L3} then for $\varepsilon>0$ small enough
  \begin{equation}
      \del{c-2M^2\varepsilon}\int_0^T|\dot{y}(t)|^2\dif t\leq \int_0^T\frac{1}{4\varepsilon}|\tilde{Q}(t)|^2\dif t.
  \end{equation}
\end{proposition}
\begin{proof} By multiplying the first equation of system \eqref{eq:E-L3} by $y(t)^T$ on the left, integrating with respect to $t$, and applying the boundary conditions we get the following equation:
   \begin{equation}
    \begin{split}
    -y(T)^TD^2g(x(T))y(T)-\int_0^T\sbr{\dot{y}(t)^TD_{vx}^2L(X(t))y(t)+\dot{y}(t)^TD_{vv}^2L(X(t))\dot{y}(t)+\dot{y}(t)^TD_{vQ}^2L(X(t))\tilde Q(t)}\dif t   \\
    =\int_0^T\sbr{y(t)^TD_{xx}^2L(X(t))y(t)+y(t)^TD_{xv}^2L(X(t))\dot{y}(t)+y(t)^TD_{xQ}^2L(X(t))\tilde Q(t)}\dif t
    \end{split}
    \end{equation}
Now by rearranging we have
\begin{equation}
    \begin{split}
        \int_0^T\sbr{y(t)^TD_{xx}^2L(X(t))y(t)+y(t)^TD_{xv}^2L(X(t))\dot{y}(t)+\dot{y}(t)^TD_{vx}^2L(X(t))y(t)+\dot{y}(t)^TD_{vv}^2L(X(t))\dot{y}(t)}\dif t\\
        +y(T)^TD^2g(x(T))y(T)=-\int_0^T\sbr{\dot{y}(t)^TD_{vQ}^2L(X(t))\tilde Q(t)+y(t)^TD_{xQ}^2L(X(t))\tilde Q(t)}\dif t.
    \end{split}
\end{equation}
Since $L$ is uniformly convex in $x$ and $v$, 
\begin{equation}
    \begin{split}
       \int_0^T\sbr{y(t)^TD_{xx}^2L(X(t))y(t)+y(t)^TD_{xv}^2L(X(t))\dot{y}(t)+\dot{y}(t)^TD_{vx}^2L(X(t))y(t)+\dot{y}(t)^TD_{vv}^2L(X(t))\dot{y}(t)}\dif t\\
       \geq c\int_0^T\del{\abs{\dot{y}(t)}^2+\abs{y(t)}^2}\dif t.
    \end{split}
\end{equation}
Also, note 
\begin{equation}
    \begin{split}
        &-\int_0^T\sbr{\dot{y}(t)^TD_{vQ}^2L(X(t))\tilde Q(t)+y(t)^TD_{xQ}^2L(X(t))\tilde Q(t)}\dif t\\
        &\leq \int_0^T\abs{\del{\dot{y}(t)^TD_{vQ}^2L(X(t))+y(t)^TD_{xQ}^2L(X(t))}\tilde Q(t)}\dif t\\
        &\leq \int_0^T \sbr{\varepsilon\abs{\dot{y}(t)^TD_{vQ}^2L(X(t))+y(t)^TD_{xQ}^2L(X(t))}^2+\frac{1}{4\varepsilon}\abs{\tilde Q(t)}^2}\dif t.
    \end{split}
\end{equation}
Putting the inequalities together, 
\begin{equation}
\begin{split}
    c\int_0^T\del{\abs{\dot{y}(t)}^2+\abs{y(t)}^2}\dif t&+ y(T)^TD^2g(x(T))y(T)\\
   &\leq \int_0^T \sbr{\varepsilon\abs{\dot{y}(t)^TD_{vQ}^2L(X(t))+y(t)^TD_{xQ}^2L(X(t))}^2+\frac{1}{4\varepsilon}\abs{\tilde Q(t)}^2}\dif t. 
\end{split}
\end{equation}
Since $M=\max\cbr{\enVert{D_{vQ}^2L(X(t))}_{L_\infty},\enVert{D_{xQ}^2L(X(t))}_{L_\infty}}<\infty$, there exists $\varepsilon>0$ such that $2M^2\varepsilon<c$. The inequality becomes
\begin{equation}\label{eq: ydot bound}
   \del{c-2M^2\varepsilon}\int_0^T\del{\abs{\dot{y}(t)}^2+\abs{y(t)}^2}\dif t+ y(T)^TD^2g(x(T))y(T)\leq \int_0^T\frac{1}{4\varepsilon}\abs{\tilde Q(t)}^2\dif t
\end{equation} 
\end{proof}

Now we can use the estimates on $\dot{y}(t)$ to prove the monotonicity of the error map.
\begin{theorem}\label{thm: G mon}
    If $L$ satisfies Assumption \ref{Lassum2} and $D^2g\geq 0$, then $E$ is monotone in the sense that for all $Q_0(t),Q_1(t)\in L^2([0,T],\bb{R}^d)$,
    \begin{equation}
        \langle E[Q_1]-E[Q_0], Q_1-Q_0\rangle\geq (1-\delta)\enVert{Q_1-Q_0}^2
    \end{equation}
    for some $\delta<1$.   
\end{theorem}
\begin{proof}
\begin{equation}
        \int_0^T\del{E[Q_1](t)-E[Q_0](t)}\cdot\del{Q_1(t)-Q_0(t)}\dif t=\int_0^T\int_0^1DE[Q_\lambda](t)\del{Q_1(t)-Q_0(t)}\dif\lambda\cdot\del{Q_1(t)-Q_0(t)}\dif t
\end{equation}
where $Q_\lambda=\lambda Q_1+(1-\lambda)Q_0$. The integral becomes
\begin{equation}
    \begin{split}
        &=\enVert{Q_1-Q_0}^2+\int_0^T\int_0^1\int_0^T\mder{\dot{x}}{Q}\del{t,x_0,Q_\lambda}(s)\del{Q_1(s)-Q_0(s)} \dif s\dif\lambda\cdot\del{Q_1(t)-Q_0(t)}\dif t\\
        &=\enVert{Q_1-Q_0}^2+\int_0^T\int_0^1\dot{y}(t,x_0,Q_\lambda)\dif\lambda\cdot\del{Q_1(t)-Q_0(t)}\dif t\\
        &\geq \enVert{Q_1-Q_0}^2-\frac{1}{2}\int_0^1\int_0^T\sbr{\abs{\dot{y}(t)}^2+\abs{Q_1(t)-Q_0(t)}^2}\dif t\dif \lambda\\
        &\geq\enVert{Q_1-Q_0}^2-\frac{1}{2}\int_0^1\int_0^T\del{1+c^*}\abs{Q_1(t)-Q_0(t)}^2\dif t\dif \lambda\\
        &=\del{1-\frac{1}{2}\del{1+c^*}}\enVert{Q_1-Q_0}^2
    \end{split}
\end{equation}
So $E$ is monotone if $c^*<1$, where $c^*=\frac{1}{4\varepsilon\del{c-2M^2\varepsilon}}$ from \eqref{eq: ydot bound}. By Assumption \ref{Lassum2}, $c^2>2M^2$, so by taking $\varepsilon = \frac{c}{4M^2}$ we get $c^* = \frac{2M^2}{c^2} < 1$.
\end{proof}
\subsection{Existence and Uniqueness}
The monotonicity of the error map, combined with standard results on uniqueness for Hamilton-Jacobi equations and continuity equations (see, for instance, \cite[Section 4]{cardaliaguet_notes_2013}), immediately implies uniqueness of solutions to the mean field game of controls \eqref{mfg}.
For completeness, in this section we collect two results that show the system is well-posed.
\begin{theorem}
    Suppose the Lagrangian $L:\bb{R}^d\times\bb{R}^d\to \bb{R}$ is $C^2$ and satisfies Assumption \ref{Lassum}, $D^2_{vv}L(v,Q)\leq MI$ for some $M<\infty$, and $D^2g\geq cI$ for some $c>0$. Then there exists a unique solution to system \eqref{mfg}.
\end{theorem}
\begin{proof}
    By Theorem \ref{thm: G mon1}, the error map $E:\bb{R}^d\to\bb{R}^d$ as defined in \eqref{error map} satisfies the condition that for all $Q_1,Q_0\in\bb{R}^d$
    \begin{equation}
         \langle E[Q_1]-E[Q_0],Q_1-Q_0\rangle>\del{1-\delta}\enVert{Q_1-Q_0}^2,
    \end{equation}
    where $\delta<1$. Notice this implies $E$ is strictly monotone and coercive. $E$ is also bounded and continuous. Applying Minty-Browder theorem yields existence and uniqueness of a fixed point $Q\in\bb{R}^d$ of \eqref{Q xdot}. Then by standard results there exists a unique solution to to system \eqref{mfg}. In particular, the Hamilton-Jacobi equation has a classical solution and the continuity equation has a solution in the sense of distributions.
\end{proof}

\begin{theorem}
    Suppose the Lagrangian $L:\bb{R}^d\times\bb{R}^d\times \bb{R}^d\to \bb{R}$ is $C^2$ and satisfies Assumption \ref{Lassum2} and $D^2g\geq 0$. Then there exists a unique solution to system \eqref{mfg}.
\end{theorem}
\begin{proof}
    By Theorem \ref{thm: G mon}, the error map $E$ as defined in \eqref{error map} satisfies the condition that for all $Q_0(t),Q_1(t)\in L^2([0,T],\bb{R}^d)$,
    \begin{equation}
        \langle E[Q_1]-E[Q_0], Q_1-Q_0\rangle\geq (1-\delta)\enVert{Q_1-Q_0}^2
    \end{equation}
    for some $\delta<1$. This implies $E$ is stictly monotone and coercive. In addition, $E$ is bounded and continuous. Applying Minty-Browder theorem yields existence and uniqueness of a fixed point $Q(t)\in L^2([0,T],\bb{R}^d)$ of \eqref{Q xdot}. Then, as in the previous argument, by standard results there exists a unique solution to to system \eqref{mfg} \cite{cardaliaguet_notes_2013}.
\end{proof}

\section{Examples and Counterexamples}\label{sec: Examples}
We use this section to emphasize the dichotomy between the monotonicity condition used in this paper with the Lasry-Lions condition and even displacement monotonicity. Due to the simplicity of this first example, we are able to explicitly solve for the optimal trajectory and then the fixed point $Q(t)$. 
\begin{example}
    Suppose $L(\dot{x},Q)=\ell(-\dot{x}+\varepsilon Q)$, $D^2\ell> 0$, and $D^2g\geq 0$, then $Q(t)$ which solves
    \begin{equation}
        Q(t) = \int_{\bb{R}^d} -\dot{x}(t;x_0,Q)\dif m_0(x_0)
    \end{equation}
    is unique.
\end{example}

    The Euler Lagrange equations for this case are:
    \begin{align}
        \od{}{t}\sbr{-D\ell(-\dot{x}(t)+\varepsilon Q(t))}&=0,\\
        x(0)&=x_0,\\
        -D\ell(-\dot{x}(T)+\varepsilon Q(T))&=-Dg(x(T)).
    \end{align}
    This implies 
    \begin{equation}
        -D\ell(-\dot{x}(t)+\varepsilon Q(t))=-Dg(x(T)).
    \end{equation}
    Then since $D^2\ell> 0$, $D\ell$ is invertible, so
    \begin{equation}
       -\dot{x}(t)+\varepsilon Q(t)=D\ell^{-1}(Dg(x(T))).
    \end{equation}
    To solve for $x(T)$, take the integral over $\sbr{0,T}$ to get
    \begin{equation}
        \del{I+TD\ell^{-1}\circ Dg}(x(T))=\int_0^T\varepsilon Q(s)\dif s+x_0.
    \end{equation}
    So
    \begin{equation}
        x(T)=\del{I+TD\ell^{-1}\circ Dg}^{-1}\del{\int_0^T\varepsilon Q(s)\dif s+x_0},
    \end{equation}
    where invertibility comes from $D^2\ell,D^2g\geq 0$. Now
    \begin{equation}
        -\dot{x}(t)=-\varepsilon Q(t)+D\ell^{-1}\circ Dg \circ \del{I+TD\ell^{-1}\circ Dg}^{-1}\del{\int_0^T\varepsilon Q(s)\dif s+x_0}.
    \end{equation}
    Next we solve for $Q$ such that 
    \begin{equation}
        Q(t) = \int_{\bb{R}^d} \sbr{-\varepsilon Q(t)+D\ell^{-1}\circ Dg \circ \del{I+TD\ell^{-1}\circ Dg}^{-1}\del{\int_0^T\varepsilon Q(s)\dif s+x_0}}\dif m_0(x_0).
    \end{equation}
    We see
    \begin{equation}
        (1+\varepsilon)Q(t) = \int_{\bb{R}^d} \sbr{D\ell^{-1}\circ Dg \circ \del{I+TD\ell^{-1}\circ Dg}^{-1}\del{\int_0^T\varepsilon Q(s)\dif s+x_0}}\dif m_0(x_0). 
    \end{equation}
    Note that the right side of the equation has no dependence on $t$, therefore $Q(t)$ must be constant. Applying this fact and rearranging we see
    \begin{equation}
    \begin{split}
      0 = \int_{\bb{R}^d} \sbr{D\ell^{-1}\circ Dg \circ \del{I+TD\ell^{-1}\circ Dg}^{-1}\del{T\varepsilon Q+x_0}-(1+\varepsilon)Q}\dif m_0(x_0)\\
      =\int_{\bb{R}^d} \sbr{-\frac{1}{T\varepsilon}\del{(1+\varepsilon)I+TD\ell^{-1}\circ Dg}\circ\del{I+TD\ell^{-1}\circ Dg}^{-1}\del{T\varepsilon Q+x_0}+\frac{1+\varepsilon}{T\varepsilon}x_0}\dif m_0(x_0)
    \end{split}
    \end{equation}
    Note $F:=\del{(1+\varepsilon)I+TD\ell^{-1}\circ Dg}\circ\del{I+TD\ell^{-1}\circ Dg}^{-1}$ is continuous, strictly monotone and coercive, so by the Minty-Browder theorem our solution is unique and we get
    \begin{equation}
        Q=\int_{\mathbb{R}^d}\frac{1}{T\varepsilon}\del{\del{I+TD\ell^{-1}\circ Dg}\circ\del{(1+\varepsilon)I+TD\ell^{-1}\circ Dg}^{-1}\del{(1+\varepsilon)x_0}-x_0}\dif m_0(x_0)
    \end{equation}
\begin{remark}
    For a counter-example showing $L(\dot{x},Q)=\ell\del{-\dot{x}+\varepsilon Q}$ is not in general Lasry-Lions monotone under the same assumptions, see Section (3.1) of \cite{graber_monotonicity_2023}.
\end{remark}
The next example uses the Cournot model. Previous papers such as \cite{graber_master_2023} have required $\varepsilon$ to be very small to guarantee uniqueness of solutions, but we show monotonicity of $E$ which give uniqueness without requiring $\varepsilon$ to be so small. 
 \begin{example}\label{ex:Cournot}
    Let  $s\in\del{-1,0}$, and let $P$ be given by
    $$P'(q)=-q^s.$$ 
    The Lagrangian $L(q,Q)=-qP(q+\varepsilon Q)$ is not Lasry-Lions monotone or displacement monotone. However, the error map $E$ from \eqref{error map} is monotone provided $\varepsilon < 2$.
\end{example}

To see that $E$ is monotone, first note the following:
\begin{equation}
	P'(q) < 0, \ P''(q) \geq 0, \ P'(q) + qP''(q)  < 0.
\end{equation}
We now see that $L$ is strictly convex in $q$: 
\begin{equation}
\begin{split}
    D^2_{qq}L(q,Q)&=-2P'(q+\varepsilon Q)-qP''(q+\varepsilon Q)\\
    &\geq -2P'(q+\varepsilon Q)-(q+\varepsilon Q)P''(q+\varepsilon Q)\\
    &>0.
\end{split}
\end{equation}
Next, we show $D^2_{pQ}H(p,Q)>-1$. Setting $q = D_p H(p,Q)$ we have
\begin{equation}
\begin{split}
D^2_{pQ}H(p,Q)&=-\frac{D_{qQ}^2 L(q,Q)}{D_{qq}^2 L(q,Q)}\\
&=-\frac{\varepsilon\del{-P'(q+\varepsilon Q)-qP''(q+\varepsilon Q)}}{-2P'(q+\varepsilon Q)-qP''(q+\varepsilon Q)}\\
&\geq -\frac{\varepsilon\del{-P'(q+\varepsilon Q)-qP''(q+\varepsilon Q)}}{2\del{-P'(q+\varepsilon Q)-qP''(q+\varepsilon Q)}}\\
&=-\frac{\varepsilon}{2}>-1.
\end{split}
\end{equation}
Now we can say by Theorem \ref{thm: G mon1}, 
    \begin{equation}
        \langle E[Q_1]-E[Q_0], Q_1-Q_0\rangle>0.
    \end{equation}

To see that $L$ is not Lasry-Lions Monotone, we show that the sign of 
\begin{equation}\label{eq:notLL}
    -\int_{q\geq 0}q\del{P\del{q+\varepsilon\int\beta\dif\mu_1(\beta)}-P\del{q+\varepsilon\int\beta\dif\mu_2(\beta)}}\dif \del{\mu_1-\mu_2}(q)
\end{equation}
changes depending on choice of $\mu_1,\mu_2$. We see \eqref{eq:notLL} becomes
\begin{equation}
    -M\int\int_0^1qP'\del{q+M_\lambda}\dif \lambda\dif\del{\mu_1-\mu_2}(q)
\end{equation}
where $M=\varepsilon\int\beta\dif\del{\mu_1-\mu_2}(\beta)$ and $M_\lambda=\varepsilon\int\beta\dif\del{\lambda\mu_1-(1-\lambda)\mu_2}(\beta)$. Now define
\begin{equation}
    F(q)=-M\int_0^1qP'\del{q+M_\lambda}\dif\lambda.
\end{equation}
Suppose $\mu_1=c_1\delta_{q_1}+(1-c_1)\delta_{\overline{q_1}}$ and $\mu_2=\delta_{q_2}$. Then \eqref{eq:notLL} becomes
\begin{equation}
    c_1F(q_1)+(1-c_1)F(\overline{q_1})-F(q_2).
\end{equation}
We see that for our choice of $P$, and $M>0$, $F$ is increasing and concave for $0<q$. Then choose $0<q_1<q_2<\overline{q_1}$ so $q_2=tq_1+(1-t)\overline{q_1}$ for some $t\in\del{0,1}$. Since $F$ is concave,
\begin{equation}
   F(q_2)>tF(q_1)+(1-t)F(\overline{q_1}). 
\end{equation}
Then for $c_1\in\del{0,t}$ 
\begin{equation}
    \begin{split}
        c_1q_1+(1-c_1)\overline{q_1}&=\overline{q_1}-c_1(\overline{q_1}-q_1)\\
        &>\overline{q_1}-t(\overline{q_1}-q_1)=q_2
    \end{split}
\end{equation}
so that $M=\varepsilon\del{c_1q_1+(1-c_1)\overline{q_1}-q_2}>0$. Now for choice of $c_1$ close enough to $t$ we get 
\begin{equation}
    c_1F(q_1)+(1-c_1)F(\overline{q_1})-F(q_2)<0.
\end{equation}
To see that \eqref{eq:notLL} could be positive, take $\mu_1=\delta_{q_1}$ and $\mu_2=\delta_{q_2}$ where $0<q_2<q_1$ so that $M>0$, then since $F$ is increasing $F(q_1)-F(q_2)>0$. 

By a similar argument we can see that $L$ is not displacement monotone. Let $\xi^1,\xi^2$ be random variables with $\mathcal{L}_{\xi^i}=\mu_i$. Supposing $Q_i=\mathbb{E}[\xi^i]=\int \xi\dif \mu_i(\xi)$, we aim to show that the sign of 
\begin{equation}\label{eq:notD}
    \mathbb{E}\sbr{\del{D_qL(\xi^1,Q_1)-D_q(\xi^2,Q_2)}\del{\xi^1-\xi^2}}
\end{equation}
changes depending on choice of $\mu_1$ and $\mu_2$. Now \eqref{eq:notD} becomes
\begin{equation}
    \mathbb{E}\sbr{\del{-\xi^1P'(\xi^1+\varepsilon Q_1)-P(\xi^1+\varepsilon Q_1)+\xi^2P'(\xi^2+\varepsilon Q_2)+P(\xi^2+\varepsilon Q_2)}\del{\xi^1-\xi^2}}
\end{equation}
Suppose for some $q_1,\overline{q_1},q_2\in\mathbb{R}$, $\mu_1=\lambda\delta_{q_1}+(1-\lambda)\delta_{\overline{q_1}}$ such that $\lambda\in(0,1)$, and $\mu_2=\delta_{q_2}$. Then define 
\begin{equation}
    F(x,y):=-xP'(x+\varepsilon y)-P(x+\varepsilon y).
\end{equation}
Then \eqref{eq:notD} becomes
\begin{equation}
    \del{\lambda F(q_1,q_\lambda)+(1-\lambda)F(\overline{q_1},q_\lambda)-F(q_2,q_2)}(q_\lambda-q_2)
\end{equation}
where $q_\lambda=\lambda q_1+(1-\lambda)\overline{q_1}$. Note $F$ is increasing and strictly concave with respect to the first variable and continuous with respect to the second variable. Fix $t\in(0,1)$, and choose $0<q_1<q_2<\overline{q_1}<R$ so that $q_2=tq_1+(1-t)\overline{q_2}$. Since $F$ is strictly concave with respect to $x$,
\begin{equation}
    F(q_2,q_\lambda) > tF(q_1,q_c)+(1-t)F(\overline{q_1},q_c).
\end{equation}
For $\lambda\in(0,t)$ notice
\begin{equation}\label{eq: q_lambda<q_2}
    \begin{split}
        \lambda q_1+(1-\lambda)\overline{q_1}&=\overline{q_1}-\lambda(\overline{q_1}-q_1)\\
        &>\overline{q_1}-t(\overline{q_1}-q_1)=q_2.
    \end{split}
\end{equation}
Now by the strict concavity of $F$ in the first variable and the continuity of $F$ in the second variable, we may choose $\lambda\in (0,t)$ close enough to $t$ such that 
\begin{equation}\label{eq: F(q_2)less than}
    F(q_2,q_2)>\lambda F(q_1,q_\lambda)+(1-\lambda)F(\overline{q_1},q_\lambda).
\end{equation}
Therefore combining the inequalities in \eqref{eq: q_lambda<q_2} and \eqref{eq: F(q_2)less than} yields
\begin{equation}
   \del{\lambda F(q_1,q_\lambda)+(1-\lambda)F(\overline{q_1},q_\lambda)-F(q_2,q_2)}(q_\lambda-q_2)<0. 
\end{equation}
Thus \eqref{eq:notD} can take on negative values.
On the other hand, if we choose any $\xi^1\neq \xi^2$ such that $Q_1=Q_2$, we get that \eqref{eq:notD} is positive because of the strict convexity of $L$ with respect to $q$.

\begin{example}\label{ex: quadratic}
    Let $L(x,v,Q)=x^2+v^2+xQ$, then $E[Q]$ is monotone, but $L$ is not Lasry-Lions monotone.
\end{example}
First, we show that $L(x,v,Q)$ satisfies the assumptions \ref{Lassum2}. 
\begin{itemize}
    \item $D^2_{vv}L(X(t))=D^2_{xx}L(X(t))=2$ and $D^2_{xv}L(X(t))=0$ so $D^2_{(x,v)}L(X(t))=2>0$.
    \item $D^2_{vQ}L(X(t))=0$ and $D^2_{xQ}L(X(t))=1$.
    \item Since $c=2$ and $M=1$, $c^2>2M^2$.
\end{itemize}
Therefore, by Theorem \ref{thm: G mon}, the error map $E$ is monotone.

However, we may show for the same $L$ that 
\begin{equation}\label{eq: LL}
    \int \del{L\del{x,v,\int a\dif\mu_1(\xi,a)}-L\del{x,v,\int a\dif\mu_2(\xi,a)}}\dif \del{\mu_1-\mu_2}(x,v)
\end{equation}
can be positive or negative depending on choice of $\mu_1$ and $\mu_2$. Letting $\mu_\lambda=\lambda\mu_1+(1-\lambda)\mu_2$, \eqref{eq: LL} becomes
\begin{equation}
    \int\int_0^1D_QL\del{x,v,\int a\dif \mu_\lambda(\xi,a)}\int a\dif\del{\mu_1-\mu_2}\dif\lambda\dif\del{\mu_1-\mu_2}(x,v)
\end{equation}
For $L(x,v,Q)=x^2+v^2+xQ$, $D_QL(x,v,Q)=x$, so it suffices to show
\begin{equation}
    \int x\dif\del{\mu_1-\mu_2}(x,v)\int v\dif\del{\mu_1-\mu_2}(x,v)
\end{equation}
can take positive or negative values. Let 
\begin{align}
    \mu_1(x,v)&=\delta_{x_1}(x)\delta_{v_1}(v)\\
    \mu_2(x,v)&=\delta_{x_2}(x)\delta_{v_2}(v).
\end{align}
Then \eqref{eq: LL} is equal to
\begin{equation}
    (x_1-x_2)(v_1-v_2).
\end{equation}
Suppose $x_1>x_2$, then \eqref{eq: LL} is positive if $v_1>v_2$ and negative if $v_1<v_2$.
\begin{remark}
    Though this example is not Lasry-Lions monotone, it is displacement monotone.
    To see this, let $X^1,X^2,V^1,V^2$ be random variables. Note $Q_i=\mathbb{E}\sbr{V^i}$. Then 
    \begin{align*}
    	\mathbb{E}&\sbr{\del{D_xL(X^1,V^1,Q_1)-D_xL(X^2,V^2,Q_2)}\del{X^1-X^2}+\del{D_vL(X^1,V^1,Q_1)-D_vL(X^2,V^2,Q_2)}\del{V^1-V^2}}\\
    	&= \mathbb{E}\sbr{2\del{X^1-X^2}^2+(Q_1-Q_2)\del{X^1-X^2}+2\del{V^1-V^2}^2}\\
    	&\geq 2\mathbb{E}\sbr{X^1-X^2}^2+\mathbb{E}\sbr{X^1-X^2}\mathbb{E}\sbr{V^1-V^2}+2\mathbb{E}\sbr{V^1-V^2}^2\\
    	&\geq \frac{1}{2}\del{\mathbb{E}\sbr{X^1-X^2}+\mathbb{E}\sbr{V^1-V^2}}^2\geq 0.
    \end{align*}
\end{remark}

\begin{example}
    As a generalization of Example \ref{ex: quadratic}, consider $L(x,v,Q)=f(x)v^2+g(x)+h(x)Q$ where $f,g,h\in C^2$. Note $D^2_{(x,v)}L\geq c>0$ if the following conditions are met.
    \begin{enumerate}
        \item There exists a $c_1>0$ such that $g''(x)\geq c_1$ for all $x\in \mathbb{R}$.
        \item $h''(x)=0$ for all $x\in\mathbb{R}$.
        \item There exists a $c_2>0$ such that $f''(x)x^2+4f'(x)x+2f(x)\geq c_2$ for all $x\in\mathbb{R}$.
    \end{enumerate}
     Additionally, if $2\enVert{h'(x)}_{L^\infty}^2<c^2$, the error map $E$ will be monotone. As in Example \ref{ex: quadratic}, we can show $L$ is not necessarily Lasry-Lions Monotone provided $h(x)$ is not constant. However, it is simple to show that under the same assumptions listed above, this example is displacement monotone. 
\end{example}

To show an $x$-dependent case where the displacement monotone condition is not met, we consider the following variant of the Cournot model.
\begin{example}
    Let $s\in(-1,0)$, $c_1<0$, and let $P$ be given by
    $$P'(q)=c_1-q^s.$$
    Then suppose $f\in C^2(\bb{R})$ is such that $f''(x)\geq c_2>0$, and $q\geq \delta>0$.
    The Lagrangian $L(x,q,Q)=-qP(q+\varepsilon Q)+f(x)$ is not displacement monotone. However, the error map $E$ from \eqref{error map} is monotone provided $\displaystyle\varepsilon<\frac{\sqrt{2}c}{-c_2+\delta^s}$, where $c:=\min\cbr{-2c_1,c_2}.$
\end{example}
First see that $L(x,q,Q)$ satisfies the Assumptions \ref{Lassum2}.
\begin{itemize}
    \item $D^2_{qq}L(x,q,Q)=-2P'(q+\varepsilon Q)-qP''(q+\varepsilon Q)\geq -2c_1$, $D^2_{xq}L(x,q,Q)=0$, and $D^2_{xx}L(x,q,Q)=f''(x)\geq c_2$, so $D^2_{(x,q)}L(x,q,Q)\geq \min\cbr{-2c_1,c_2}=c>0.$
    \item $D^2_{xQ}L(x,q,Q)=0$ and $D^2_{qQ}L(x,q,Q)=\varepsilon\del{-P'(q+\varepsilon Q)-qP''(q+\varepsilon Q)}\leq \varepsilon(-c_1+\delta^s)$. Thus $M=\varepsilon(-c_1+\delta^s)<\infty$.
    \item $M^2<2c^2$ for $\displaystyle\varepsilon<\frac{\sqrt{2}c}{-c_2+\delta^s}$.
\end{itemize}
Therefore, by Theorem \ref{thm: G mon1} the error map $E$ from \eqref{error map} is monotone. To see that the Lagrangian is not displacement monotone, we must show that the sign of 
\begin{equation}\label{eq:notD2}
    \mathbb{E}\sbr{\del{D_qL(X^1,\xi^1,Q_1)-D_q(X^1,\xi^2,Q_2)}\del{\xi^1-\xi^2}+\del{D_xL(X^1,\xi^1,Q_1)-D_x(X^1,\xi^2,Q_2)}\del{X^1-X^2}}
\end{equation}
changes depending on choice of random variables $X^1,X^2,\xi^1,\xi^2$. If $X^1=X^2$, \eqref{eq:notD2} becomes
\begin{equation}
        \mathbb{E}\sbr{\del{-\xi^1P'(\xi^1+\varepsilon Q_1)-P(\xi^1+\varepsilon Q_1)+\xi^2P'(\xi^2+\varepsilon Q_2)+P(\xi^2+\varepsilon Q_2)}\del{\xi^1-\xi^2}}.
\end{equation}
Note that we showed this could take on positive or negative values for a similarly defined $P$ in Example \ref{ex:Cournot}, and the $P$ in this example has the same qualities required. Thus, the Lagrangian is not displacement monotone.

It is also simple to show that this example is not Lasry-Lions Monotone.
	\bibliographystyle{alpha}
	\bibliography{mybib}
\end{document}